\documentclass[11pt]{amsart}
\usepackage[T1]{fontenc}
\usepackage[utf8]{inputenc}

\usepackage{amsmath,amssymb,amsthm,amsfonts,mathtools,bm}

\usepackage{lmodern}
\usepackage[final]{microtype}
\usepackage{babel}
\usepackage{nicefrac} 

\usepackage[margin=1.2in]{geometry}

\usepackage{multicol}
\usepackage{array}
\usepackage{booktabs}
\usepackage[shortlabels]{enumitem}

\usepackage{tikz}
\usetikzlibrary{cd}
\usetikzlibrary{shapes}
\usetikzlibrary{shapes.multipart}

\usepackage{thmtools}
\usepackage{thm-restate}
\usepackage[hyphens]{url}
\usepackage[colorlinks]{hyperref}
\hypersetup{
    colorlinks,
    linkcolor={red!50!black},
    citecolor={blue!50!black},
    urlcolor={blue!80!black}
}

\theoremstyle{plain}
\newtheorem{theorem}{Theorem}[section]

\newtheorem*{claim*}{Claim}

\newtheorem*{proposition*}{Proposition}
\newtheorem{fact}[theorem]{Fact}
\newtheorem*{fact*}{Fact}
\newtheorem{conjecture}[theorem]{Conjecture}
\newtheorem*{conjecture*}{Conjecture}

\newtheorem{lemma}[theorem]{Lemma}
\newtheorem*{lemma*}{Lemma}

\newtheorem*{question*}{Question}
\theoremstyle{definition}
\theoremstyle{definition}\newtheorem*{remark*}{Remark}
\theoremstyle{definition}\newtheorem{definition}[theorem]{Definition}
\theoremstyle{definition}\newtheorem*{definition*}{Definition}
\theoremstyle{definition}
\theoremstyle{definition}
\theoremstyle{definition}
\theoremstyle{definition}\newtheorem*{example*}{Example}

\newcommand{\term}{\textbf} 
\newcommand\seq[1]{\langle #1 \rangle} 
\renewcommand{\bar}{\overline} 
\renewcommand{\phi}{\varphi} 
\renewcommand{\epsilon}{\varepsilon} 

\DeclareMathOperator{\dom}{dom} 

\def\N{\mathbb{N}} 
\def\R{\mathbb{R}} 
\def\Cantor{2^\omega} 

\def\ZF{\mathsf{ZF}}
\def\ZFC{\mathsf{ZFC}}

\def\AD{\mathsf{AD}}

\def\Unif{\mathsf{Uniformization}}
\def\CH{\mathsf{CH}}

\def\PSP{\mathsf{PSP}}
\def\H{\mathbb{H}}
\def\c{\mathfrak{c}}

\def\LN{\mathsf{LN}}
\DeclareMathOperator{\head}{head}
\DeclareMathOperator{\tail}{tail}

\begin{document}

\author{Kojiro Higuchi}
\address{College of Engineering, Nihon University, 1 Nakagawara, Tokusada, Tamuramachi, Koriyama, 963-8642, Japan}
\email{higuchi.koujirou@nihon-u.ac.jp}

\author{Patrick Lutz}
\address{Department of Mathematics, University of California, Los Angeles}
\email{pglutz@math.ucla.edu}

\title[A Note on a Conjecture of Sacks]{A Note on a Conjecture of Sacks: It is Harder to Embed Height Three Partial Orders than Height Two Partial Orders}

\begin{abstract}
A long-standing conjecture of Sacks states that it is provable in $\ZFC$ that every locally countable partial order of size continuum embeds into the Turing degrees. We show that this holds for partial orders of height two, but provide evidence that it is hard to extend this result even to partial orders of height three. In particular, we show that the result for height two partial orders holds both in certain extensions of $\ZF$ with only limited forms of choice and in the Borel setting (where the partial orders and embeddings are required to be Borel measurable), but that the analogous result for height three partial orders fails in both of these settings. We also formulate a general obstacle to embedding partial orders into the Turing degrees, which explains why our particular proof for height two partial orders cannot be extended to height three partial orders, even in $\ZFC$. We finish by discussing how our results connect to the theory of countable Borel equivalence relations.
\end{abstract}

\maketitle
\vspace{-20pt}
\thispagestyle{empty}

\section{Introduction}

An enduring goal of computability theory is to determine which structures can be embedded into the Turing degrees. When the structures under consideration are partial orders, there are two obvious restrictions: since the set of Turing degrees has size continuum and every Turing degree has at most countably many predecessors, any partial order which embeds into the Turing degrees must also have these properties. A famous conjecture of Sacks states that these are the only restrictions.

More precisely, say that a partial order $(P, \leq_P)$ is \term{locally countable} if every element $x \in P$ has at most countably many predecessors (i.e.\ the set $\{y \in P \mid y <_P x\}$ is countable). In 1963, Sacks conjectured that every locally countable partial order of size continuum can be embedded into the Turing degrees \cite{sacks1963degrees}. Sacks himself proved that this holds in $\ZFC + \CH$ (by showing that it holds in $\ZFC$ for all locally countable partial orders of size $\omega_1$), but whether it is provable in $\ZFC$ alone is still unknown.

We will not resolve Sacks's conjecture in this paper. Instead, we will present a curious phenomenon related to it. Namely, we will demonstrate that it is easy to embed partial orders of height two into the Turing degrees, but hard to embed partial orders of height three. 
 
This statement deserves some explanation. First, we will prove (in $\ZFC$) that every locally countable partial order of size continuum and height two embeds into the Turing degrees. We will also show that this result is robust, in the sense that it holds even in settings where only weak forms of choice are available. 

Second, we will show that there is a general obstacle to embedding partial orders into the Turing degrees which implies that our method of embedding height two partial orders cannot be extended to partial orders of height three. Essentially, our method for height two partial orders embeds the first level of the partial order as a perfect set, but we show that whenever the image of an embedding contains a perfect set, the embedding cannot be extended very much. Moreover, in some settings this obstacle actually yields an outright proof that not all locally countable partial orders of size continuum and height three can be embedded into the Turing degrees, including some of the settings in which our proof for height two partial orders works. Thus, there are some settings in which all locally countable partial orders of size continuum and height two embed into the Turing degrees, but the same does not hold for height three.

We will show that this is the case in two particular settings: the $\ZF$ setting, where we work in certain extensions of $\ZF$ which contradict the full Axiom of Choice (but still satisfy weak forms of it), and the Borel setting, where the partial orders and embeddings are required to be Borel measurable.\footnote{It has been observed before that there are many similarities between the $\ZF$ setting and the Borel setting. One example can be found in the work of Shani on Borel equivalence relations~\cite{shani2021borel}.} At the end of the paper, we will mention some connections between our results and the theory of countable Borel equivalence relations.

We will now give more precise statements of these results (in particular, we will specify which extensions of $\ZF$ we consider and what we mean by the ``Borel setting'') and discuss the context for our obstacle to embedding partial orders into the Turing degrees.


\subsection{Partial orders of finite height}

For any natural number $n > 0$, a partial order has \term{height \boldmath$n$} if its longest chain has length exactly $n$. We can think of such a partial order as consisting of $n$ ``levels'': the first level consists of those elements of the partial order with no predecessors, the second level consists of those elements whose predecessors are all in the first level (and which have at least one predecessor in the first level), the third level consists of those elements whose predecessors are all in the first two levels (and which have at least one predecessor in the second level), and so on. The fact that there are no chains of length greater than $n$ implies that every element will end up in one of these $n$ levels (and the fact that there is some chain of length $n$ implies that no level is empty). A typical partial order of height three, stratified into levels in this manner, looks something like this:
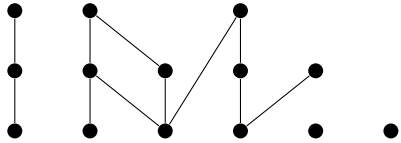
\begin{figure}[ht]
\centering
\begin{tikzpicture}
    \foreach \i in {0,1,...,5}
    { 
        \node[black, circle, fill, inner sep=2pt] (a\i) at (\i, 0) {};
    }
    \foreach \i in {0,1,...,4}
    { 
        \node[black, circle, fill, inner sep=2pt] (b\i) at (\i, 0.8) {};
    }
    \foreach \i in {0,1,3}
    { 
        \node[black, circle, fill, inner sep=2pt] (c\i) at (\i, 1.6) {};
    }
    \foreach \x/\y in {a0/b0, b0/c0, a1/b1, b1/c1, a2/b1, a2/b2, a3/b3, b3/c3, a2/c3, b2/c1, a3/b4}
    {
        \draw (\x) -- (\y);
    }
\end{tikzpicture}
\end{figure}


\noindent As stated above, we will prove that Sacks's conjecture holds for partial orders of height two.

\begin{restatable}{theorem}{heighttwozfc}
\label{thm:height2zfc}
Every locally countable partial order of size continuum and height two embeds into the Turing degrees.
\end{restatable}

An equivalent statement was also proved by Kumar and Raghavan in~\cite{kumar2021separating}, by a somewhat different technique. However, for reasons that we will discuss later, their proof does not generalize to other settings as well as ours.

\subsection{Obstacles to embedding partial orders in the Turing degrees}

Suppose you have a partial order $(P, \leq_P)$ of size continuum and you want to embed $P$ into the Turing degrees. A reasonable approach is to pick a well-ordering of $P$ of length continuum and define an embedding by transfinite recursion. In other words, pick up elements of $P$ one at a time and show that as long as you have embedded fewer than continuum many elements so far, there is always a place to map the next element to. This is essentially the approach taken by Sacks to embed locally countable partial orders of size $\omega_1$ in~\cite{sacks1963degrees}.

A fundamental obstacle to using this approach to solve Sacks's Conjecture was discovered by Groszek and Slaman~\cite{groszek1983independence}. Say that a set $A$ of Turing degrees is \term{Turing independent} if there is no finite subset of $A$ whose join computes some element of $A$ not in the subset. Groszek and Slaman proved that it is consistent with $\ZFC$ that there is a maximal Turing independent set of size less than continuum. Thus if you want to construct an embedding by transfinite recursion, you have to be careful not to end up with this particular Turing independent set in the image of your embedding at any step in the recursion. For suppose that you do. If you later encounter another element of $P$ which is sufficiently independent of all the elements you have seen so far, there will be nowhere to map it to. Kumar has used Groszek and Slaman's technique to show that a similar problem may occur even when embedding a partial order of height three whose first level has size $\omega_1$~\cite{kumar2019suborders}.

One potential solution to this problem is to make the transfinite recursion satisfy some stronger inductive assumption that prevents it from accidentally building any set like the one constructed by Groszek and Slaman, but no such condition has been identified so far. 

A different potential solution is to take a more structural approach. To illustrate what we mean, suppose that the partial order $P$ which we want to embed has finite height. Instead of embedding the elements of $P$ one at a time, we can first find an especially nice subset of the Turing degrees to map the first level of $P$ to, then use the niceness of this subset to find another nice subset to map the second level to, and so on. 

In fact, this is exactly the approach we use to prove Theorem~\ref{thm:height2zfc}. In particular, we embed the elements of the first level of the partial order as the Turing degrees of a Turing independent perfect set of reals. Since we don't rely on transfinite recursion, this approach is easier to generalize to settings with only limited forms of choice.

It might seem reasonable to hope that our proof can be generalized to deal with partial orders of any finite height. For example, we might try to show that the second level of $P$ can also be embedded as a Turing independent perfect set, and more generally that if the $n^\text{th}$ level can be embedded as a Turing independent perfect set then so can the $(n + 1)^\text{st}$ level. However, our next theorem shows that this is not possible.

\begin{restatable}{theorem}{obstacleone}
\label{thm:obstacle1}
There is a locally countable partial order $(P, \leq_P)$ of size continuum and height three which has the following properties.
\begin{enumerate}
    \item The first level of $P$ has size continuum.
    \item If $f$ is any function from $P$ into the Turing degrees such that the image of $f$ on the first level of $P$ contains a perfect set then $f$ is not an embedding.
\end{enumerate}
\end{restatable}

The fist condition on $P$ may appear somewhat arbitrary, but it is necessary to make the second condition nontrivial: if the first level of $P$ has size less than continuum then its image under $f$ cannot contain a perfect set for cardinality reasons and so the second condition is vacuously true.

It is common to phrase obstacles to embedding partial orders into the Turing degrees in terms of obstacles to extending Turing independent sets. We can also do that here. Note that when we say that a set of reals $A$ is Turing independent, we mean that no finite subset of $A$ computes any other element of $A$.

\begin{restatable}{theorem}{obstacletwo}
\label{thm:obstacle2}
Suppose $A$ is a Turing independent set of reals which contains a perfect set, $A'$, $B$ is a countable, dense subset of $A'$ and $x$ computes every element of $B$. Then $(A\setminus B)\cup \{x\}$ is not Turing independent.
\end{restatable}

To prove Theorems~\ref{thm:obstacle1} and~\ref{thm:obstacle2}, we rely on a technical theorem on perfect sets, previously used by the second author and Benjamin Siskind in work on Martin's Conjecture~\cite{lutz2023part}.

\subsection{The \texorpdfstring{$\ZF$}{ZF} setting}

By examining the proof of Theorem~\ref{thm:height2zfc} above, we will see that it does not require the full Axiom of Choice. In particular, let us consider two weak choice principles: \term{Uniformization for Reals}, denoted $\Unif_\R$, and \term{Lusin-Novikov Choice}, denoted $\LN$.

\begin{itemize}
    \item $\Unif_\R$ states that every real-indexed family of nonempty sets of reals has a choice function---i.e.\ if $R$ is a binary relation on $\Cantor$ such that for each $x$, $\{y \mid R(x, y)\}$ is nonempty, then there is some function $f\colon \Cantor \to \Cantor$ such that for each $x$, $R(x, f(x))$ holds.
    \item $\LN$ states that for every binary relation $R$ on $\Cantor$, if every section of $R$ is countable (i.e.\ for each $x\in \Cantor$, $\{y \in \Cantor \mid R(x, y)\}$ is countable) then there is a function $f\colon \Cantor \to (\Cantor)^{\leq \omega}$ enumerating the elements of each section.
\end{itemize}

It is not hard to check that $\Unif_\R$ implies $\LN$ (the point is just that the set of enumerations of a countable set of reals can itself be thought of as a set of reals). We will see that it is basically trivial to modify the proof of Theorem~\ref{thm:height2zfc} to work in $\ZF + \Unif_\R$ and, with only slightly more care, it is also possible to modify the proof to work in $\ZF + \LN$. Thus we have the following theorem.

\begin{restatable}[$\ZF + \LN$]{theorem}{heighttwozf}
\label{thm:height2zf}
Every locally countable partial order of size continuum and height two embeds into the Turing degrees.
\end{restatable}

On the other hand, there is also an extension of $\ZF$ in which we can use Theorem~\ref{thm:obstacle1} to show that not every locally countable partial order of size continuum and height three can be embedded into the Turing degrees. Let $\PSP$ denote the \term{Perfect Set Principle}, which states that every subset of $\Cantor$ is either countable or contains a perfect subset. The key point is that in $\ZF + \PSP$, the hypothesis of Theorem~\ref{thm:obstacle1} is automatically satisfied. Using this observation, we will prove the following theorem.

\begin{restatable}[$\ZF + \PSP$]{theorem}{heightthreezf}
\label{thm:height3zf}
There is a locally countable partial order of size continuum and height three which does not embed into the Turing degrees.
\end{restatable}

Note that the theory $\ZF + \LN + \PSP$, in which Theorems~\ref{thm:height2zf} and~\ref{thm:height3zf} are both provable, is known to be consistent and thus there is a single consistent extension of $\ZF$ in which Sacks's conjecture holds for partial orders of height two but not for partial orders of height three. 

Here's one way to see that $\ZF + \LN + \PSP$ is consistent. Let $\AD_\R$ denote the \term{Axiom of Real Determinacy}, an axiom which has been extensively studied in inner model theory~\cite{solovay1978independence}. $\ZF + \AD_\R$ implies the Axiom of Determinacy, and hence $\PSP$ (see \cite{jech2003set}, Theorem 33.3) as well as $\Unif_\R$~\cite{solovay1978independence}, and hence $\LN$. Furthermore, $\ZF + \AD_\R$ is known to be consistent, assuming large cardinals.

As a side note, Theorem~\ref{thm:height3zf} shows that Sacks's conjecture is independent of $\ZF$. As far as we are aware, this fact has not been published before.\footnote{Also note that unlike the theory $\ZF + \LN + \PSP$, which is known to have greater consistency strength than $\ZF$, $\ZF + \PSP$ is equiconsistent with $\ZF$ by a result of Truss (see Theorem 3.2 of~\cite{truss1974models}, though note that Glazer has identified a mistake in that paper which will be fixed in his upcoming thesis~\cite{glazer2023personal}).}

\subsection{The Borel setting}

Statements analogous to Theorems~\ref{thm:height2zf} and~\ref{thm:height3zf} also hold in the Borel setting, where we consider only Borel partial orders and Borel embeddings.

A \term{Borel partial order} is simply a partial order $(P, \leq_P)$ such that $P$ and $\leq_P$ are Borel measurable subsets of $\Cantor$ and $\Cantor\times \Cantor$, respectively.\footnote{One may instead assume that $P$ is a standard Borel space and $\leq_P$ is a Borel measurable subset of $P \times P$; this does not make a difference for any of the results of this paper.} If $(P, \leq_P)$ is a Borel partial order then a \term{Borel embedding} of $P$ into Turing reducibility is a Borel measurable function $f\colon P \to \Cantor$ such that for all $x, y \in P$,
\[
    x \leq_P y \iff f(x) \leq_T f(y).
\]
One oddity here is that Turing reducibility itself is not a Borel partial order according to our definition (because it is not a partial order on $\Cantor$, but only a quasi-order). We will return to this point in section~\ref{sec:cber}.

It is relatively straightforward to modify the proofs of Theorems~\ref{thm:height2zf} and~\ref{thm:height3zf} to yield the following (essentially we just replace $\LN$ and $\PSP$ with analogous theorems provable in the Borel setting).

\begin{restatable}{theorem}{heighttwoborel}
\label{thm:height2borel}
Every locally countable Borel partial order of height two has a Borel embedding into Turing reducibility.
\end{restatable}

\begin{restatable}{theorem}{heightthreeborel}
\label{thm:height3borel}
There is a locally countable Borel partial order of height three with no Borel embedding into Turing reducibility.
\end{restatable}

Note that in this context, we do not need to explicitly assume that $P$ has size continuum---this follows automatically from the definition of ``Borel partial order.''

\subsection*{Acknowledgements}

Thanks to Steffen Lempp for encouraging us to write this paper and to Benny Siskind and Ted Slaman for several helpful conversations. Also thanks to Ashutosh Kumar for pointing us to the papers \cite{groszek1983independence, kumar2019suborders, kumar2021separating} and for an interesting email exchange and to Elliot Glazer for answering several questions about choiceless set theory.

\section{Embedding height two partial orders is easy}
\label{sec:height2}

In this section we will explain how to embed any height two, locally countable partial order of size continuum into the Turing degrees. As discussed in the introduction, we will first prove the theorem in $\ZFC$ and then explain how to modify the proof to work in the theory $\ZF + \LN$ and in the Borel setting.

Here's the basic strategy of the proof. Given a partial order $(P, \leq_P)$ of height two, we will construct a function $f\colon P \to \Cantor$ such that
\[
    x \leq_P y \iff f(x) \leq_T f(y).
\]
It is clear that such a function induces an embedding of $P$ into the Turing degrees, so this is sufficient. In order to construct $f$, we will first pick a perfect set of mutually generic reals. We will then map each element of the first level of $P$ to an element of this perfect set and each element of the second level of $P$ to a sufficiently generic upper bound of the images of its predecessors. There is one wrinkle in this proof: we need to ensure that even when two elements of the second level have exactly the same predecessors, they get mapped to incomparable Turing degrees.\footnote{A similar problem arises if an element of the second level has exactly one predecessor: we need to make sure it gets sent to a different Turing degree than its predecessor.} One way to handle this is to insert a unique point below each element of the second level, which is not below any other elements of the partial order. This ensures that the elements of the second level all have distinct sets of predecessors and an embedding of this new partial order yields an embedding of the original partial order by forgetting about the new elements that we added. When we give the proof in detail, we will not quite explain things in this way, but it is essentially what will happen.


\subsection{\texorpdfstring{$\ZFC$}{ZFC} case}

We will break the proof into two lemmas, the first of which tells us that we can find a perfect set of mutually generic reals to map the elements of the first level to and the second of which tells us that we can always find sufficiently generic upper bounds to map the elements of the second level to. Both lemmas are essentially folklore, though we are not aware of anywhere that they are written up in the precise form we would like to use. To state these two lemmas, we need to recall the definition of a perfect tree.

\begin{definition}
A tree $T \subseteq 2^{< \omega}$ is called a \term{perfect tree} if every node in $T$ has incomparable descendants in $T$.
\end{definition}

The name ``perfect tree'' is used because the set of infinite paths through $T$ is a perfect subset of $2^\omega$. A perfect tree $T$ can be pictured as a kind of warped version of $2^{< \omega}$: there are no dead ends and if you follow any path for long enough, you will eventually come to a place where you can choose to go left or right and remain in the tree either way. In $2^{< \omega}$ you can make this decision on every step while in an arbitrary perfect tree you may have to take many steps in between each decision. By making this picture more precise, it is possible to prove that the set of paths through a perfect tree is always homeomorphic to $2^\omega$.

\begin{fact}
If $T \subseteq 2^{< \omega}$ is a perfect tree then $[T]$ is homeomorphic to $2^\omega$ and therefore has size continuum.
\end{fact}

Instead of stating our first lemma in terms of mutually generic reals we will state it using the notion of a Turing independent set, which was mentioned in the introduction (and which is all we really need for our application of the lemma). Actually, we will not quite use the definition of ``Turing independent set'' from the introduction, but an essentially equivalent notion which applies to sets of reals rather than sets of Turing degrees.

\begin{definition}
A set $A \subseteq 2^\omega$ of reals is called a \term{Turing independent set} if no finite subset of $A$ computes any other element of $A$---i.e.\ if $a_0, \ldots, a_n \in A$ and $b$ is any element of $A$ not equal to any $a_i$ then $a_0\oplus \ldots \oplus a_n$ does not compute $b$.
\end{definition}

We can now actually state our first lemma. The argument is due to Sacks~\cite[Theorem 3]{sacks1961suborderings}, though he did not state it in terms of perfect trees.

\begin{lemma}[Sacks]
\label{lemma:tree}
There is a perfect tree $T$ such that $[T]$ is a Turing independent set.
\end{lemma}

\begin{proof}
First consider how one might construct a perfect tree with no computable branches. To do this, we need to ensure that every branch of the tree disagrees with each total computable function in at least one place. We can accomplish this by ``growing'' the tree from the root node up in a series of stages. At each stage we have built a finite tree and we continue growing it by extending the leaf nodes (i.e.\ by adding children to the leaf nodes, children to those children, and so on). On alternate stages we can add incomparable children below every leaf node (to make sure the tree is perfect) and extend each leaf node to make sure any branch which extends it disagrees with the next total computable function (to make sure no branch is computable). Note that we do not need the tree itself to be computable so these steps are easy to carry out.

To make sure that no finite set of branches computes any other branch, we can do something similar but now instead of extending leaf nodes one at a time to make them disagree with the next computable function, we need to extend finite sets of leaf nodes at the same time to make sure the next computable function which uses those branches as an oracle disagrees with the branches of the tree which extend the other leaf nodes.

We will now describe this a bit more formally. We will form a sequence of finite subtrees of $2^{< \omega}$, $T_0 \subseteq T_1 \subseteq T_2 \subseteq \ldots$ such that $T_{n + 1}$ is an end extension of $T_n$ (every node in $T_{n + 1} \setminus T_n$ extends a leaf node of $T_n$). The final tree will be obtained as $T = \bigcup_{n \in \N} T_n$. We can start with $T_0$ as the tree just consisting of a single root node and nothing else (i.e.\ just the empty sequence).

Now we will explain how to extend $T_n$ to $T_{n + 1}$. The idea, again, is to first split every leaf node in $T_n$ and then extend all of them without splitting in order to make sure that no finite subset of them is correctly computing any of the others using the the $n^\text{th}$ Turing functional. To this end, first let $T_n^0$ be the tree formed by adding incomparable children below each leaf node of $T_n$. In other words,
\[
T_n^0 = T_n \cup \{\sigma^\frown 0 \mid \sigma \text{ is a leaf node of $T_n$}\} \cup \{\sigma^\frown 1 \mid \sigma \text{ is a leaf node of $T_n$}\}.
\]

Next, let $\Phi$ be the $n^\text{th}$ Turing functional in some standard enumeration. We will form a finite sequence of end extensions $T_n^0 \subseteq T_n^1 \subseteq \ldots \subseteq T_n^k$ (where $k$ is the number of nonempty subsets of the set of leaves of $T_n^0$) and take $T_{n + 1} = T_n^k$. In each of these extensions, we will not split any nodes. In other words, each leaf of $T_n^0$ will have at most one descendant at each level of $T_n^i$.

Suppose we have already formed $T_n^i$ and let $S$ be the $i^\text{th}$ nonempty subset of the set of leaves of $T_n^0$. We will now explain how to form $T_n^{i + 1}$. Our goal is to ensure that no set of branches extending the nodes in $S$ can compute any branch extending any other leaf node of $T_n^0$.

Since we never split any nodes in any of the previous extensions of $T_n^0$, each element of $S$ corresponds to a unique leaf of $T_n^i$. Let $\sigma_1, \ldots, \sigma_l$ denote these leaves of $T_n^i$. Let $N$ be a number larger than the height of $T_n^i$.  Now either we can find extensions $\tau_1,\ldots, \tau_l$ of $\sigma_1,\ldots,\sigma_l$ such that $\Phi^{\tau_1\oplus\ldots\oplus\tau_l}(N)$ converges or we can't find such extensions. In the former case, define $T_n^{i + 1}$ from $T_n^i$ by extending each $\sigma_j$ to $\tau_j$ and extending all other leaf nodes of $T_n^i$ to strings whose $N^\text{th}$ bit disagrees with $\Phi^{\tau_1\oplus\ldots\oplus\tau_l}(N)$. In the latter case, set $T_n^{i + 1} = T_n^i$.

Now let's check that $[T]$ is really a Turing independent set. Let $x_1,\ldots,x_l$ and $y$ be distinct elements of $[T]$ and let $\Phi$ be any Turing functional. Let $n$ be some number large enough that all of $x_1,\ldots, x_l$ and $y$ correspond to distinct leaf nodes in $T_n$ and chosen so that the $n^\text{th}$ Turing functional is equivalent to $\Phi$ (we are assuming that every computable function shows up infinitely often in whatever enumeration we are using). Suppose that $x_1,\ldots, x_l$ correspond to the $i^\text{th}$ set of leaves of $T_n^0$. Then our definition of $T_n^{i + 1}$ ensures that either $\Phi^{x_1\oplus \ldots\oplus x_l}$ disagrees with $y$ in at least one place (this corresponds to the first case in our construction above) or $\Phi^{x_1\oplus\ldots\oplus x_l}$ is not a total function (this corresponds to the second case).
\end{proof}

Our second lemma guarantees we can find sufficiently generic upper bounds to map the elements of the second level of our partial order to. The idea of the proof is originally due to Spector, who used it to show that every increasing sequence of Turing degrees has an exact pair of upper bounds (see Theorem 6.5.3 of \cite{soare2016turing}).

\begin{lemma}
\label{lemma:upperbound}
Suppose $T$ is a perfect tree such that $[T]$ is Turing independent. Then every countable subset of $[T]$ has an upper bound in the Turing degrees which does not compute any other element of $[T]$.
\end{lemma}

\begin{proof}
Suppose $A$ is a countable subset of $[T]$ and $x_0, x_1, x_2, \ldots$ is an enumeration of the elements of $A$. In order to uniformly handle both the case where $A$ is finite and the case where $A$ is infinite (which will be helpful later) we allow the enumeration to contain repetitions and therefore can assume that it is always infinite.

Here's the idea of the proof. We will construct an element $y$ of $2^{\omega \times \omega}$ such that column $n$ of $y$ consists of some finite string followed by $x_n$ and this $y$ will be the upper bound we are after. It is easy to see that any such $y$ computes each element of $A$ and so the bulk of the proof consists of showing that if we choose the finite strings in a sufficiently generic way then $y$ does not compute any other element of $[T]$. The proof crucially depends on the fact that $[T]$ is Turing independent.

Formally, we will construct $y$ in a series of stages. At the end of stage $n$ we will have constructed a finite list of finite strings, $\sigma_1, \ldots, \sigma_k$ (where $k$ may not be equal to $n$) and on stage $n + 1$ we will add some more strings onto the end of this list. At the end, we will define column $i$ of $y$ to be $\sigma_i^\frown x_i$. The idea is that on stage $n + 1$ we will ensure that if we run the $n^\text{th}$ program with oracle $y$ it either does not compute any element $[T]$ or it computes one of $x_1,\ldots, x_k$.

We will now explain how to complete one step of this construction. Suppose we have just completed stage $n$ and our list of finite strings is $\sigma_0,\sigma_1,\ldots,\sigma_k$. Let us say that a finite string $\tau$ in $2^{<\omega \times < \omega}$ (i.e.\ a finite initial segment of an element of $2^{\omega\times\omega}$) \term{agrees with $y$ so far} if for each $i \leq k$, column $i$ of $\tau$ agrees with $\sigma_i^\frown x_i$. In other words, $\tau$ is a possible initial segment of $y$ given what we have built by the current stage. Let $\Phi$ denote the $n^\text{th}$ Turing functional. There are four cases to consider.

\medskip
\noindent\textbf{Case 1:} There is some finite string $\tau \in 2^{< \omega \times < \omega}$ which agrees with $y$ so far such that $\Phi^\tau$ cannot be extended to a path through $T$. In this case, extend the list $\sigma_0,\sigma_1,\ldots,\sigma_k$ to ensure that $y$ is an extension of $\tau$. This guarantees that $\Phi^y$ is not in $[T]$.

\medskip
\noindent\textbf{Case 2:} There is some finite string $\tau$ which agrees with $y$ so far and some $m \in \N$ such that for every extension $\tau'$ of $\tau$ which agrees with $y$ so far, $\Phi^{\tau'}(m)$ does not converge. In this case, extend the list $\sigma_0,\sigma_1,\ldots,\sigma_k$ to ensure that $y$ is an extension of $\tau$. This guarantees that $\Phi^y$ is not total.

\medskip
\noindent\textbf{Case 3:} For every finite string $\tau$ which agrees with $y$ so far, $\Phi^\tau$ is compatible with one of $x_0,x_1,\ldots,x_k$. In this case, do nothing; we are already guaranteed that $\Phi^y$ is either not total or is equal to one of $x_0, x_1, \ldots, x_k$.

\medskip
\noindent\textbf{Case 4:} None of the first three cases holds. We claim this case actually cannot happen. In particular, in this case we can use $x_0, x_1, \ldots, x_k$ to compute another element of $[T]$, thus violating our assumption that $[T]$ is Turing independent.

To do so, first note that because Case 3 does not hold, we can find some finite string $\tau$ which agrees with $y$ so far and such that $\Phi^\tau$ is not compatible with any of $x_0, x_1,\ldots,x_k$. Next, inductively form a sequence $\tau = \tau_0 \prec \tau_1 \prec \tau_2 \prec \ldots$ of finite strings (which all agree with $y$ so far) as follows. Given $\tau_m$, look for some extension $\tau_{m + 1}$ of $\tau_m$ which agrees with $y$ so far and such that $\Phi^{\tau_{m + 1}}(m)$ converges. Because we are not in Case 2, we will always be able to find such a string.

Let $z$ be the infinite sequence formed by the $\tau_m$. By construction, $\Phi^z$ is total. Since $z$ extends $\tau$, $\Phi^z$ is not equal to any of $x_0,\ldots,x_k$. Since Case 1 does not hold, $\Phi^{\tau_m}$ is compatible with an element of $[T]$ for each $m$. Since $[T]$ is closed, this implies that $\Phi^z$ itself is in $[T]$.

To summarize, $\Phi^z$ is an element of $[T]$ which is not equal to any of $x_0,\ldots, x_k$. The final point is that to carry out the process of choosing the $\tau_m$ described above, we only need to be able to check which finite strings agree with $y$ so far. If we know $x_0,x_1,\ldots,x_k$ then this is easy to do, so $z$ (and hence $\Phi^z$) is computable from $x_0\oplus x_1\oplus\ldots\oplus x_k$.
\end{proof}

We will now explain how to put these two lemmas together to prove Theorem~\ref{thm:height2zfc}.

\heighttwozfc*

\begin{proof}
Let $(P, \le_P)$ be a height two, locally countable partial order of size continuum. Let $P_1$ denote the first level of $P$ and let $P_2$ denote the second level. As we mentioned above, instead of directly defining a map of $P$ into the Turing degrees, we will define a map $f\colon P \to \Cantor$ such that for all $x, y \in P$,
\[
    x \leq_P y \iff f(x) \leq_T f(y).
\]


\medskip\noindent\textbf{Definition of $\bm{f}$.}
First, let $T$ be a perfect tree such that $[T]$ is Turing independent, as in Lemma \ref{lemma:tree}. Since $P$ and $[T]$ both have size continuum, we can find an injective map $g \colon P \to [T]$.

We will define $f(x)$ by cases depending on whether $x$ is in $P_1$ or $P_2$. If $x$ is in $P_1$ then we simply set $f(x) = g(x)$. If $x$ is in $P_2$ then define $f(x)$ as follows. Let $P_{\le x} = \{y \in P \mid y \le_P x\}$ be the set of (non-strict) predecessors of $x$ in $P$; note that $P_{\le x}$ includes $x$ itself. By Lemma \ref{lemma:upperbound}, we can find a real which computes every element of $g(P_{\le x})$ but which computes no other elements of $[T]$. Set $f(x)$ equal to some such real.

\medskip\noindent\textbf{$\bm{f}$ is an embedding.}
Now we need to check that $f$ is an embedding. Let $x$ and $y$ be any two distinct elements of $P$. We need to show that $x \le_P y$ if and only if $f(x) \le_T f(y)$. 

First suppose $x \leq_P y$. If $x = y$ then we are done, and if not then $x$ must be in the first level of $P$ and $y$ must be in the second level. Therefore $f(x) = g(x)$ and $f(y)$ is an upper bound in the Turing degrees for a set which includes $g(x)$, so $f(y)$ computes $f(x)$.

Now suppose that $x \nleq_P y$. We know that no matter which level $x$ is in, $f(x)$ computes $g(x)$. So to show that $f(y)$ doesn't compute $f(x)$, it is enough to show that $f(y)$ doesn't compute $g(x)$. If $y$ is in the first level of $P$ then this is guaranteed by the fact that $g(x)$ and $g(y)$ are distinct elements of the Turing independent set $[T]$ and $f(y) = g(y)$. And if $y$ is in the second level of $P$ then since $x$ is not a predecessor of $y$, our choice of $f(y)$ ensures that it cannot compute $g(x)$.
\end{proof}

The theorem we have just proved is very similar to Theorem 2.2 of the paper ``Separating Families and Order Dimension of Turing Degrees'' by Kumar and Raghavan~\cite{kumar2021separating}. That theorem states that a specific height two, locally countable partial order of size continuum---which the authors refer to as $\H_\c$---embeds into the Turing degrees.\footnote{Kumar and Raghavan~\cite{kumar2021separating} also pointed out that both $\H_\c$ and the Turing degrees have the largest order dimension among all locally countable partial orders of size continuum. On the order dimension of the Turing degrees, see also \cite{higuchi2020order}.} Obviously this theorem is implied by our Theorem~\ref{thm:height2zfc}. On the other hand, it is not too hard to show that every height two, locally countable partial order of size continuum embeds into $\H_\c$ and so the two theorems are actually equivalent.

However, there are some differences between our proof and that of Kumar and Raghavan. While we embedded the first level of the partial order as a single Turing independent perfect set and then embedded the elements of the second level essentially independently of each other, Kumar and Raghavan construct their embedding by a transfinite recursion of length continuum. To find places to embed elements of the first level of the partial order, they use the existence of a Turing independent set of size continuum and to find places for elements of the second level, they use the fact that for any countable ideal of Turing degrees there is a set of reals of size continuum, any two of which form an exact pair for the ideal. 

Kumar and Raghavan's approach has some advantages and disadvantages compared to our approach. On the one hand, since they do not rely on any specific properties of perfect sets their approach does not obviously fall prey to the obstacle presented by our Theorem~\ref{thm:obstacle1} (and there are known constructions of large Turing independent sets which do not produce perfect sets, see~\cite{kumar2023large}). On the other hand, since their approach relies on picking a well-order of size continuum, it cannot be easily adapted to work in the $\ZF$ or Borel settings and also seems more susceptible to the obstacle discovered by Groszek and Slaman (see the note~\cite{kumar2019suborders} by Kumar for more about how this problem applies to their approach). 

In the end, for all their differences, both methods run into the same problem when embedding partial orders of height three: it's not clear how to make sure the second level of the partial order is embedded as a Turing independent set.

\subsection{\texorpdfstring{$\ZF + \LN$}{ZF + LN} case}

We will now show how to modify our proof of Theorem~\ref{thm:height2zfc} to work in the theory $\ZF + \LN$. In other words, we will explain how to prove the following theorem.

\heighttwozf*

The key observation is that there is actually only one part of the proof that cannot be carried out in $\ZF$: the application of Lemma~\ref{lemma:upperbound}. 

Recall that in the proof of Theorem~\ref{thm:height2zfc}, we had a height two, locally countable partial order $(P, \leq_P)$ of size continuum with first level $P_1$ and second level $P_2$. We first picked a perfect tree $T$ such that $[T]$ is Turing independent and an injection $g\colon P \to [T]$. Up to this point in the proof, everything works in $\ZF$. Then, for each $x \in P_2$, we used Lemma~\ref{lemma:upperbound} to find an upper bound for $g(P_{\leq x})$ which does not compute any other element of $[T]$. 

The problem is that while Lemma~\ref{lemma:upperbound} itself is provable in $\ZF$, it only implies the existence of an appropriate upper bound for $g(P_{\leq x})$, but does not tell us an explicit way to choose such an upper bound. Since we need to choose an upper bound for each $x \in P_2$, mere existence is not enough and so this part of the proof seems to require the use of some form of choice. It is fairly clear that $\Unif_\R$ is enough: we can identify the elements of $P_2$ with reals and the set of appropriate upper bounds for each $g(P_{\leq x})$ with a set of reals and thus $\Unif_\R$ lets us pick an appropriate upper bound for each $g(P_{\leq x})$.

However, by examining the proof of Lemma~\ref{lemma:upperbound}, we can do slightly better. The only reason the proof of that lemma does not give us an explicit way of constructing an upper bound is because it requires choosing an enumeration of the given countable set. Thus if we could choose an enumeration of each set $g(P_{\leq x})$ then we could use the proof of Lemma~\ref{lemma:upperbound} to pick upper bounds without any further use of choice. Fortunately, choosing enumerations of countable sets of reals is exactly what $\LN$ lets us do.

\subsection{Borel case}

We will now explain how to modify the proof of Theorem~\ref{thm:height2zfc} to work in the Borel setting. In other words, how to prove the following theorem.

\heighttwoborel*

Since the changes are slightly more substantial than for the $\ZF + \LN$ case, this time we will give a formal proof. But first, we will explain what changes we need to make. There are essentially two problems to overcome in this setting.

The first problem is that, just as in the $\ZF + \LN$ case, we need to deal with the one part of the proof that uses choice. Fortunately, there is an analogue of Lusin-Novikov Choice in the Borel setting: the Lusin-Novikov Theorem.\footnote{Which is why we chose the name ``Lusin-Novikov Choice'' in the first place.} See Theorem 18.10 and Exercise 18.15 of Kechris's book~\cite{kechris1995classical} for a proof.

\begin{theorem}[Lusin-Novikov Uniformization Theorem]
Suppose $R$ is a Borel subset of $2^\omega \times 2^\omega$ with countable sections (i.e.\ for each $x \in 2^\omega$, the set $\{y \mid R(x, y)\}$ is countable). Then both of the following hold:
\begin{enumerate}
    \item The domain of $R$ (i.e.\ the set $\dom(R) = \{x \mid \exists y\, R(x, y)\}$) is Borel.
    \item There is a sequence $\seq{f_n}_{n \in \N}$ of Borel functions $f_n \colon \dom(R) \to \Cantor$ enumerating the sections of $R$. (i.e.\ for each $x \in \dom(R)$, $\{f_n(x) \mid n \in \N\} = \{y \mid R(x, y)\}$).
\end{enumerate}
\end{theorem}

The second problem is that in our proof of Theorem~\ref{thm:height2zfc}, we defined the embedding by cases depending on whether the input was in the first level or second level of the partial order. In $\ZF$, this is not a problem, but in the Borel setting, this will only yield a Borel function if each of the two cases corresponds to a Borel subset of the partial order. However, it turns out that we can use the Lusin-Novikov Theorem to solve this problem as well. 

When giving the formal proof, it will be helpful to have an explicitly stated, refined version of Lemma~\ref{lemma:upperbound} (which guarantees the existence of sufficiently generic upper bounds). The proof of the refined version of the lemma just consists of noting that other than choosing an enumeration of the countable set, every part of the construction of the upper bound in the original proof was arithmetically definable and thus yields a Borel map from enumerations to upper bounds.

\begin{lemma}
\label{lemma:upperboundborel}
Suppose $T$ is a perfect tree such that $[T]$ is Turing independent. Then there is a Borel function $h_T \colon (2^\omega)^\omega \to 2^\omega$ such that for any sequence $\bar{x} = \langle x_n \rangle_{n \in \N}$ of elements of $[T]$, $h_T(\bar{x})$ computes all of the $x_n$'s but does not compute any other element of $[T]$.
\end{lemma}

\begin{proof}[Proof of Theorem~\ref{thm:height2borel}]
Let $(P, \le_P)$ be a height two, locally countable Borel partial order and let $P_1$ and $P_2$ denote the first and second levels of $P$, respectively. Recall that in the proof of Theorem~\ref{thm:height2zfc}, we constructed a map $f \colon P \to 2^\omega$ such that $x \le_P y$ if and only if $f(x) \le_T f(y)$. In this proof, we will simply give a definition of the function $f$ which makes it clear why it is Borel; the proof that it is an embedding is unchanged.

Let $T$ be a perfect tree such that $[T]$ is Turing independent and let $h_T$ be the function from Lemma~\ref{lemma:upperboundborel}. Let $g \colon 2^\omega \to [T]$ be a homeomorphism (though we really only need $g$ to be a Borel bijection). Now consider the binary relation $R$ on $\Cantor$ defined by
\[
    R(x, y) \iff y <_P x.
\]
Note that since $\leq_P$ is Borel, so is $R$. Furthermore, note that the domain of $R$ is exactly $P_2$ and for each $x \in P_2$, the section $\{y \mid R(x, y)\}$ is exactly the set of (strict) predecessors of $x$. Since $P$ is locally countable, the sections of $R$ are all countable. Thus we can apply the Lusin-Novikov Theorem to conclude two things:
\begin{enumerate}
    \item $P_2$ is Borel. Since $P$ itself is Borel, this implies that $P_1 = P\setminus P_2$ is Borel as well.
    \item There are Borel functions $k_n \colon P_2 \to \Cantor$ such that for each $x \in P_2$, $\{k_n(x) \mid n \in \N\}$ is equal to the set of (strict) predecessors of $x$.
\end{enumerate}
It is then straightforward to check that the following definition:
\[
    f(x) =
    \begin{cases}
        g(x) &\text{ if } x \in P_1\\
        h_T(\seq{g(x), g(k_0(x)), g(k_1(x)), \ldots}) &\text{ if } x \in P_2. 
    \end{cases}
\]
yields a Borel function.
\end{proof}

\section{Embedding height three partial orders is hard}

Theorem~\ref{thm:obstacle1} from the introduction gives a general obstacle to embedding height three partial orders in the Turing degrees. In this section, we will give a proof of that theorem and then explain how it yields unconditional non-embeddability results in $\ZF + \PSP$ and in the Borel setting. First, though, let's recall the statement of Theorem~\ref{thm:obstacle1}.

\obstacleone*

There is one thing we should explain about the statement of this theorem. The function $f$ in the statement is a function from the partial order $P$ to the set of Turing degrees. Thus the image of $f$ on the first level of $P$ is a set of Turing degrees, not a set of reals, so it cannot literally contain a perfect set.  What we really mean is that there is a perfect set of reals, $A$, such that the Turing degree of each element of $A$ is in the image of $f$ on the first level of $P$.

\subsection{Proof of Theorem~\ref{thm:obstacle1}}

The main tool we will use in the proof of Theorem~\ref{thm:obstacle1} is the following technical theorem on perfect sets, which was proved in~\cite{lutz2023part}.

\begin{theorem}[$\ZF$]
\label{thm:perfect}
Suppose that $A$ is a perfect subset of $\Cantor$, $B$ is a countable dense subset of $A$ and $b$ is a real which computes each element of $B$. Then for every $c \in \Cantor$, there are reals $d_1, d_2, d_3, d_4 \in A$ such that
\[
    b\oplus d_1 \oplus d_2 \oplus d_3 \oplus d_4 \geq_T c.
\]
\end{theorem}

To prove Theorem~\ref{thm:obstacle1}, we will assume that we have an arbitrary partial order $(P, \leq_P)$ with the properties listed in the theorem statement, along with an embedding $f$ of $P$ into the Turing degrees. We will then assume that the image of $f$ on the first level of $P$ contains a perfect set and try to use Theorem~\ref{thm:perfect} to derive a contradiction; the main idea is that Theorem~\ref{thm:perfect} puts certain constraints on what configurations of points can be realized in the Turing degrees, but an arbitrary partial order does not necessarily have these constraints. In order to make the proof work, we will make some assumptions about the structure of $P$. We will keep track of these assumptions and, at the end of the proof, check that there is a partial order which satisfies all of them.

To begin, let $(P, \leq_P)$ be a partial order of height three and let $P_1, P_2$ and $P_3$ denote the first, second and third levels of $P$, respectively. Implicit in the statement of Theorem~\ref{thm:obstacle1} we already have some assumptions about $P$.
\begin{enumerate}[start = 1, label=(A\arabic*), leftmargin=43pt]
    \item $P$ is locally countable, size continuum and height three.
    \item $P_1$ has size continuum.
\end{enumerate}
Next, let $f$ be an embedding of $P$ into the Turing degrees and assume there is a perfect set $A \subseteq \Cantor$ such that every element of $A$ has Turing degree in $f(P_1)$. We will now attempt to derive a contradiction by finding a configuration of points in $P$ which $f$ cannot take to an isomorphic configuration in the Turing degrees. We will break our construction into a sequence of steps. To follow the construction, it may help to refer to Figure~\ref{fig:config}, which depicts the final configuration.

\begin{figure}
    \centering
    \begin{tikzpicture}
        \draw (1,0) ellipse (3cm and 0.5cm) node[left=100] (P1) {$P_1$};
        \draw (1,1.1) ellipse (3cm and 0.4cm) node[left=100] (P2) {$P_2$};
        \draw (1,2.1) ellipse (3cm and 0.4cm) node[left=100] (P3) {$P_3$};
        \draw (9,0) ellipse (3cm and 0.5cm) node[right=90] (A) {$A$};

        \draw[->] (4.5,1) -- (6, 1) node[above, midway] (f) {$f$};

        \draw (7.7, 0) ellipse (1.3cm and 0.25cm) node[below=3, left=15] (B) {};
        \node (Blabel) at (6, -0.75) {$B$};
        \draw (Blabel) -- (B);
        \draw (-0.3, 0) ellipse (1.3cm and 0.25cm) node[below=3, left=6] (fB) {};
        \node (fBlabel) at (-1.5, -0.85) {$f^{-1}(B)$};
        \draw (fBlabel) -- (fB);
        \foreach \i in {0,1,...,3}
        {   \node[black, circle, fill, inner sep=1pt] (p\i) at (\i/3 - 1.2, 0) {};
            \node[black, circle, fill, inner sep=1pt] (b\i) at (\i/3 + 6.7, 0) {};  }
        \node (petc) at (0.2, 0) {$\ldots$};
        \node (betc) at (8.1, 0) {$\ldots$};

        \node[label=left:$x$, black, circle, fill, inner sep=1pt] (p) at (-0.8, 1.1) {};
        \node[label=left:$b$, black, circle, fill, inner sep=1pt] (b) at (7.2, 1.1) {};
        \node[label=right:$y$, black, circle, fill, inner sep=1pt] (q) at (2.9, 1.1) {};
        \node[label=left:$w$, black, circle, fill, inner sep=1pt] (s) at (0.4, 2.1) {};
        \node[label=right:$c$, black, circle, fill, inner sep=1pt] (c) at (10.9, 1.1) {};
        \node[label=left:$e$, black, circle, fill, inner sep=1pt] (e) at (8.4, 2.1) {};
        \foreach \i in {1,2,...,4}
        {   \node[label=below:$z_{\i}$, black, circle, fill, inner sep=1pt] (r\i) at (\i*0.4 + 1, 0.2) {};
            \node[label=below:$d_{\i}$, black, circle, fill, inner sep=1pt] (d\i) at (\i*0.4 + 9, 0.2) {}; }

        \foreach \i in {0,1,...,3}
        {   \draw (p) -- (p\i);
            \draw (b) -- (b\i); }
        \draw (s) -- (p);
        \draw (e) -- (b);
        \draw[dashed] (e) -- (c);
        \foreach \i in {1,2,...,4}
        {   \draw (s) -- (r\i);
            \draw (e) -- (d\i); }
    \end{tikzpicture}
    \caption{The sets and points chosen during the proof of Theorem~\ref{thm:obstacle1} and their relationships to each other.}
    \label{fig:config}
\end{figure}
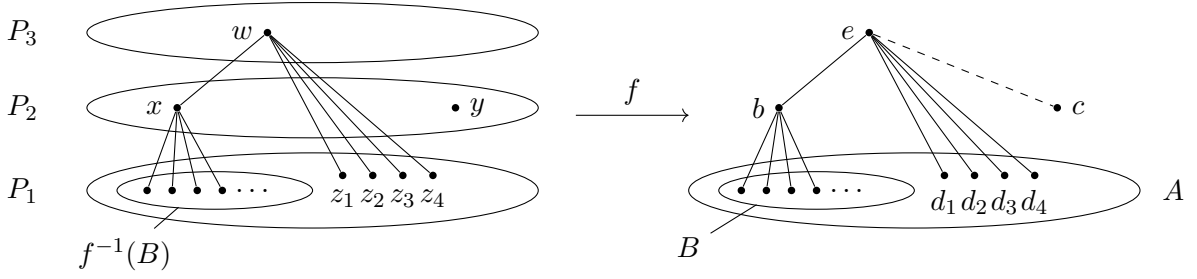

\medskip\noindent\textbf{Step 1: Pick a countable dense set.} Let $B$ be a countable, dense subset of $A$. Let $f^{-1}(B)$ denote the subset of $P$ consisting of points which are in the preimage of the set of Turing degrees of reals in $B$. Note that $f^{-1}(B)$ is a countable subset of $P_1$.

\medskip\noindent\textbf{Step 2: Pick an upper bound for \boldmath$B$.} Let $x$ be an element of $P_2$ which is an upper bound for $f^{-1}(B)$. To ensure that such an element exists, we will make another assumption about $P$.
\begin{enumerate}[resume*]
    \item Every countable subset of $P_1$ has an upper bound in $P_2$.
\end{enumerate}
Now let $b$ be any real whose Turing degree is $f(x)$. Note that since $f$ is an embedding of $P$ into the Turing degrees, $b$ computes each element of $B$.

\medskip\noindent\textbf{Step 3: Pick an independent element of \boldmath$P$.} Let $y$ be an element of $P_2$ which is not equal to $x$. Note that the existence of such an element follows from the assumptions we have made so far: by (A3), each countable subset of $P_1$ has an upper bound in $P_2$, and since $P$ is locally countable and $P_1$ is uncountable, no single element can be an upper bound for all of these countable subsets. Let $c$ be a real whose Turing degree is $f(y)$.

\medskip\noindent\textbf{Step 4: Apply Theorem~\ref{thm:perfect}.} We are now in position to apply Theorem~\ref{thm:perfect}. By that theorem, there are reals $d_1, d_2, d_3, d_4$ in $A$ such that $b\oplus d_1\oplus d_2\oplus d_3\oplus d_4 \geq_T c$. Since $d_1,\ldots,d_4$ are in $A$, their Turing degrees are in $f(P_1)$. Thus there are elements $z_1,z_2,z_3,z_4 \in P_1$ such that for each $i$, $f(z_i)$ is the Turing degree of $d_i$.

\medskip\noindent\textbf{Step 5: Reach a contradiction.} Let $w$ be an element of $P_3$ such that $w$ is above $x$ and $z_1,z_2,z_3,z_4$ but \emph{not} above $y$. To ensure that such an element exists, we will make our final assumption about $P$.
\begin{enumerate}[resume*]
    \item For every finite subset $Q$ of $P_1\cup P_2$ and every element $q \in P_2$ which is not equal to any element of $Q$, there is an element of $P_3$ which is above every element of $Q$, but not above $q$.
\end{enumerate}
Let $e$ be a real whose Turing degree is $f(w)$. Since $f$ is an embedding, $e$ computes $b$ and each of $d_1,d_2,d_3,d_4$. Hence we have
\[
    e \geq_T b \oplus d_1\oplus d_2\oplus d_3\oplus d_4 \geq_T c.
\]
However, this contradicts the fact that $w \ngeq_P y$, which, since $f$ is an embedding, should imply that $e \ngeq_T c$.

To complete the proof, we just have to show that there is a partial order $P$ which satisfies the assumptions listed above.

\begin{lemma}
\label{lemma:height3exists}
There is a partial order $(P, \leq_P)$ such that
\begin{enumerate}
    \item $P$ is locally countable, size continuum and height three.
    \item $P_1$ has size continuum.
    \item Every countable subset of $P_1$ has an upper bound in $P_2$.
    \item For every finite subset $Q$ of $P_1\cup P_2$ and every element $q \in P_2$ which is not equal to any element of $Q$, there is an element of $P_3$ which is above every element of $Q$, but not above $q$.
\end{enumerate}
Where $P_1, P_2$, and $P_3$ denote the first, second and third levels of $P$, respectively.
\end{lemma}

\begin{proof}
Essentially, we can take the ``free'' locally countable order of height three and size continuum. The construction is not hard, and can be explained very succinctly, but in order to make it clear that the resulting partial order is Borel (which is necessary to prove Theorem~\ref{thm:height3borel} below), we will give a more involved definition. It will help to first fix some notation.

We will construct $P$ as a subset of $3^\omega$. Given $x \in 3^\omega$, let $\head(x)$ denote the first digit of $x$ and let $\tail(x)$ denote the element of $3^\omega$ given by deleting the first digit of $x$. It will be useful to view some elements of $P$ as coding a countable sequence of elements of $3^\omega$. Given $x \in 3^\omega$, we will use $x_n$ to denote the $n^\text{th}$ element of the sequence coded by $\tail(x)$. 

We will now define the partial order $(P, \leq_P)$. First, we define the three levels of $P$ as follows.
\begin{align*}
    P_1 &= \{x \in 3^\omega \mid \head(x) = 0\}\\
    P_2 &= \{x \in 3^\omega \mid \head(x) = 1 \text{ and for some $n$, } x_n \in P_1\}\\
    P_3 &= \{x \in 3^\omega \mid \head(x) = 2 \text{ and for some $n$, } x_n \in P_2\}
\end{align*}
and we let $P = P_1\cup P_2\cup P_3$. Next, we define the order $\leq_P$ on $P$. For $x, y \in P$, set $x \leq_P y$ if and only if one of the following conditions holds.
\begin{enumerate}
    \item $x = y$.
    \item $x$ is in a lower level than $y$ and for some $n$, $x = y_n$.
    \item $x$ is in $P_1$, $y$ is in $P_3$ and there are some $n, m$ such that $y_n \in P_2$ and $x = (y_n)_m$.
\end{enumerate}
Note that the last condition is needed to make $\leq_P$ transitive.

It is straightforward to check that $P$ is size continuum, locally countable and has height three and that $P_1$ also has size continuum. It is also straightforward to check that $P_1, P_2,$ and $P_3$ are indeed the first, second and third levels of $P$, respectively. To verify the other two properties required of $P$, we will check the following more general property, which implies both of them:
\begin{quote}
    If $Q$ is a countable, downwards-closed subset of $P$ such that the maximum level of any element of $Q$ is $i \leq 2$ then there is an element of $P_{i + 1}$ whose set of strict predecessors is exactly $Q$.
\end{quote}
To prove this, simply note that we can take the element $x \in P$ such that $\head(x) = i + 1$ and $\tail(x)$ codes a countable sequence enumerating the elements of $Q$.
\end{proof}

\subsection{\texorpdfstring{$\ZF + \PSP$}{ZF + PSP} case}

We will now see how to use Theorem~\ref{thm:obstacle1} to prove Theorem~\ref{thm:height3zf}.

\heightthreezf*

The first key point is that in $\ZF + \PSP$ we can use the perfect set property to show that the second condition in Theorem~\ref{thm:obstacle1} is always satisfied. The second key point is that the proof of Theorem~\ref{thm:obstacle1} still works in $\ZF$. To see this, it suffices to note the following.
\begin{enumerate}
    \item Theorem~\ref{thm:perfect} is provable in $\ZF$.
    \item The construction of the partial order $P$ in Lemma~\ref{lemma:height3exists} still works in $\ZF$ and $P$ still has the claimed properties.
    \item It is provable in $\ZF$ that every perfect subset of $\Cantor$ has a countable, dense subset and all other choices made during the proof of Theorem~\ref{thm:obstacle1} only involve instantiating a finite number of existential quantifiers.
\end{enumerate}
We can now prove Theorem~\ref{thm:height3zf}. Let $(P, \leq_P)$ be the partial order in the statement of Theorem~\ref{thm:obstacle1}. Suppose for contradiction that $f$ is an embedding of $P$ into the Turing degrees. Let $P_1$ denote the first level of $P$ and let $A$ denote the set of reals whose Turing degree is in $f(P_1)$. Note that since $P_1$ is uncountable, so is $A$. Now recall that the axiom $\PSP$ states that every uncountable set of reals contains a perfect set. Since we are working in $\ZF + \PSP$, the set $A$ must contain a perfect set. By our choice of $P$ (from the statement of Theorem~\ref{thm:obstacle1}), this implies that $f$ is not an embedding, which is a contradiction.

\subsection{Borel case}

We will now use Theorem~\ref{thm:obstacle1} to prove Theorem~\ref{thm:height3borel}.

\heightthreeborel*

The proof is similar to the $\ZF + \PSP$ case; the key point is that we can replace the axiom $\PSP$ with the perfect set theorem for analytic sets.

\begin{theorem}[Perfect set theorem for analytic sets; \cite{kechris1995classical} Exercise 14.13]
Every $\mathbf{\Sigma^1_1}$ subset of $2^\omega$ is either countable or contains a perfect set.
\end{theorem}

There are two other key points. First, the partial order $P$ we constructed in the proof of Lemma~\ref{lemma:height3exists} is Borel and second, $P_1$, the first level of $P$, is also Borel. These can both be seen by inspecting the proof. 

Granting this, we can prove Theorem~\ref{thm:height3borel} as follows. Let $f\colon P \to \Cantor$ be a Borel function and assume for contradiction that for all $x, y \in P$,
\[
    x \leq_P y \iff f(x) \leq_T f(y).
\]
In other words, $f$ induces an embedding of $P$ into the Turing degrees, which we will denote $\widetilde{f}$ (i.e.\ $\widetilde{f}(x)$ is the Turing degree of $f(x)$). Since $P_1$ and $f$ are both Borel, $f(P_1)$ is analytic. And since $P_1$ is uncountable and $f$ is injective, $f(P_1)$ is uncountable. Hence by the perfect set theorem for analytic sets, $f(P_1)$ contains a perfect set. Thus $\widetilde{f}$ satisfies the hypothesis of Theorem~\ref{thm:obstacle1} and so it is not an embedding, contradicting our assumption.

\subsection{Other nonembedding results}

Using the same techniques we used to prove the main theorems of this section, we can prove a few other related results. Since the proofs do not contain any new ideas, we will keep them brief. We will begin with Theorem~\ref{thm:obstacle2} from the introduction.

\obstacletwo*

\begin{proof}
We can divide $A'$ into two disjoint perfect sets, $A_0'$ and $A_1'$, such that $B\cap A_0'$ is a countable, dense subset of $A_0'$ (for example, we can take $A_0'$ to be the intersection of $A'$ with some basic open neighborhood of $\Cantor$). Let $y$ be any element of $A_1'\setminus B$. We will show that $x$, together with a finite number of elements of $A_0'\setminus B$, computes $y$, which is enough to show that $(A\setminus B)\cup \{x\}$ is not Turing independent.

By Theorem~\ref{thm:perfect}, we can find reals $z_1,z_2,z_3, z_4$ in $A_0'$ such that $x\oplus z_1\oplus z_2\oplus z_3\oplus z_4$ computes $y$. This is almost enough, except that some of the $z_i$'s could be in $B$. However, if they are then $x$ computes them by assumption, so we can leave them out.
\end{proof}

We will also give an example of another nonembedding result that can be proved using our techniques. Recall that a subset $Q$ of a partial order $(P, \leq_P)$ is \term{countably directed} if every countable subset of $Q$ has an upper bound in $Q$.

\begin{theorem}[$\ZF + \PSP$]
Suppose $P$ is a locally countable partial order, $Q$ is an uncountable, countably directed subset of $P$ and $x$ is an element of $P$ which is not below any element of $Q$. Then $P$ cannot be embedded into the Turing degrees.
\end{theorem}

\begin{proof}
Suppose for contradiction that $f$ is an embedding of $P$ into the Turing degrees. Since $Q$ is uncountable, so is $f(Q)$. Thus by $\PSP$, the set of reals whose Turing degree is in $f(Q)$ must contain a perfect set, $A$. Let $B$ be a countable, dense subset of $A$. Let $y$ be an element of $Q$ such that $f(y)$ is above the Turing degree of each element of $B$ (such a $y$ exists because $Q$ is countably directed). Let $b$ be a real whose Turing degree is $f(y)$ and note that $b$ computes every element of $B$. Let $c$ be a real whose Turing degree is $f(x)$. By Theorem~\ref{thm:perfect}, we can find reals $d_1,d_2,d_3,d_4$ in $A$ such that $b\oplus d_1\oplus d_2\oplus d_3\oplus d_4 \geq_T c$. Since the $d_i$ are all in $A$, we can find $z_1, z_2, z_3, z_4 \in Q$ such that $f(z_i)$ is the Turing degree of $d_i$. Again using the fact that $Q$ is countably directed, we can find an upper bound $w$ in $Q$ for $y, z_1,\ldots,z_4$. Thus $f(w) \geq_T f(x)$, but since $f$ is an embedding, this contradicts the fact that $x$ is not below any element of $Q$.
\end{proof}

Essentially the same proof can be used to prove the Borel version of this result (note the extra assumption that $Q$ is Borel, which is needed to apply the perfect set theorem).

\begin{theorem}
Suppose $P$ is a locally countable Borel partial order of size continuum, $Q$ is an uncountable, countably directed Borel subset of $P$ and $x$ is an element of $P$ which is not below any element of $Q$. Then $P$ has no Borel embedding into Turing reducibility.
\end{theorem}

\section{Countable Borel equivalence relations}
\label{sec:cber}

There is something a little odd about our results in the Borel setting. Namely, Theorems~\ref{thm:height2borel} and~\ref{thm:height3borel} are about whether or not certain Borel partial orders on the reals have Borel embeddings into Turing reducibility. But according to our definitions, Turing reducibility itself is not a Borel partial order. Instead, Turing reducibility (as a relation on $\Cantor$) is a Borel \emph{quasi-order}. This suggests that in the Borel setting, we should consider which Borel quasi-orders embed into the Turing degrees. And since the theory of locally countable Borel quasi-orders in some ways parallels the more well-studied theory of countable Borel equivalence relations, if we phrase our results in this way then it seems natural to compare them to what is known about countable Borel equivalence relations.

\subsection{Countable Borel equivalence relations.}
We will begin by reviewing a few definitions; for a more thorough introduction, see the recent survey by Kechris~\cite{kechris2021theory}. A \term{countable equivalence relation} is simply an equivalence relation whose equivalence classes are all countable. A \term{countable Borel equivalence relation} is a countable equivalence relation $(X, \sim_X)$ such that $X$ is a Borel subset of $\Cantor$ and $\sim_X$ is a Borel subset of $\Cantor\times\Cantor$. A \term{Borel reduction} from $(X, \sim_X)$ to $(Y, \sim_Y)$ is a Borel function $f\colon X \to Y$ such that $x \sim_X y \iff f(x) \sim_Y f(y)$. 

One of the main focuses of the theory of countable Borel equivalence relations is to determine which countable Borel equivalence relations are Borel reducible to each other. An important role in the theory is played by the \term{universal countable Borel equivalence relations}. Briefly, a countable Borel equivalence relation is universal if every other countable Borel equivalence relation is Borel reducible to it. Several countable Borel equivalence relations are known to be universal---for example, the orbit equivalence relation of the shift action of the free group on two generators~\cite{dougherty1994structure} and arithmetic equivalence~\cite{marks2016martins}. Kechris has conjectured that Turing equivalence is also universal~\cite{dougherty2000how}.

\begin{conjecture}[Kechris]
Turing equivalence is a universal countable Borel equivalence relation.
\end{conjecture}

This conjecture is interesting both for its own sake and because it contradicts Martin's conjecture, a major open question in computability theory; see \cite{marks2016martins} for more about the connection between the two conjectures.

\subsection{Locally countable Borel quasi-orders.}
The theory of countable Borel equivalence relations has a natural analogue in the theory of locally countable Borel quasi-orders. A \term{quasi-order} is simply a transitive, reflexive binary relation (essentially a partial order where some elements are allowed to be equivalent to each other) and a quasi-order $(P, \leq_P)$ is \term{locally countable} if for every $x \in P$, the set $\{y \mid y \leq_P x\}$ is countable. 

In analogy with the definitions above, a \term{locally countable Borel quasi-order} is a locally countable quasi-order $(P, \leq_P)$ such that $P$ is a Borel subset of $\Cantor$ and $\leq_P$ is a Borel subset of $\Cantor\times\Cantor$. A \term{Borel reduction}\footnote{Note that earlier we spoke of \emph{Borel embeddings} whereas here we say \emph{Borel reductions}. In descriptive set theory (and in particular, in the theory of Borel equivalence relations and Borel quasi-orders) these have distinct meanings. However, they have the same meaning when the domain is a Borel partial order, which justifies our use of ``Borel embedding'' earlier.} from $(P, \leq_P)$ to $(Q, \leq_Q)$ is a Borel function $f\colon P \to Q$ such that $x \leq_P y \iff f(x) \leq_Q f(y)$ and a locally countable Borel quasi-order is \term{universal} if every other locally countable Borel quasi-order is Borel reducible to it.

There is a close connection between countable Borel equivalence relations and locally countable Borel quasi-orders. First, any equivalence relation $(X, \sim_X)$ literally is a quasi-order and if $\sim_X$ is countable as an equivalence relation then it is locally countable as a quasi-order. Second, given a quasi-order $(P, \leq_P)$, there is an associated equivalence relation $\sim_P$ on $P$, defined by
\[
    x \sim_P y \iff x\leq_P y \text{ and } y \leq_P x.
\]
If $\leq_P$ is locally countable then $\sim_P$ is countable and if $\leq_P$ is Borel then so is $\sim_P$. Moreover, universality of $\leq_P$ and $\sim_P$ usually go together: it is typically the case that if $\leq_P$ is universal among locally countable Borel quasi-orders then $\sim_P$ is universal among countable Borel equivalence relations, and vice-versa. For example, arithmetic reducibility is a univeral locally countable Borel quasi-order and arithmetic equivalence, its associated equivalence relation, is a universal countable Borel equivalence relation.

\subsection{Kechris's conjecture and quasi-orders.}
As we mentioned above, Turing reducibility, considered as a relation on $\Cantor$, is a locally countable Borel quasi-order. Also, its associated equivalence relation is just Turing equivalence. In light of this, and of the discussion above, there is a natural analogue of Kechris's conjecture for locally countable Borel quasi-orders---namely, the statement that Turing reducibility is a universal locally countable Borel quasi-order.


Since every locally countable Borel partial order is also a locally countable Borel quasi-order (and since a Borel reduction of a Borel partial order into a Borel quasi-order is automatically a Borel embedding), Theorem~\ref{thm:height3borel} shows that this statement is false (this was also proved by related means in~\cite{lutz2023part}). Since, as we have already mentioned, universality among locally countable Borel quasi-orders and among countable Borel equivalence relations seem to be strongly correlated, this theorem is evidence against Kechris's conjecture.

It is also possible to formulate a question that is partway between Kechris's conjecture for countable Borel equivalence relations and for locally countable Borel quasi-orders. Define the \term{height} of a quasi-order $(P, \leq_P)$ as the length of the longest strictly decreasing chain in $P$ (i.e.\ the height of the partial order formed by quotienting $P$ by $\sim_P$). Earlier we said that a countable Borel equivalence relation literally is a locally countable Borel quasi-order. Note that this quasi-order always has height one. Also note that if $(X, \sim_X)$ is a countable Borel equivalence relation then a Borel reduction of $\sim_X$ (as an equivalence relation) to Turing equivalence is not quite the same as a Borel reduction of $\sim_X$ (as a quasi-order) to Turing reducibility: the former just needs to send distinct $\sim_X$-equivalence classes to distinct Turing degrees whereas the latter needs to send them to \emph{incomparable} Turing degrees. Thus, one can view the statement that every locally countable Borel quasi-order of height one is Borel reducible to Turing reducibility as a mild strengthening of Kechris' conjecture.

In this paper, we have not just shown that Kechris's conjecture is false when countable Borel equivalence relations are replaced by locally countable Borel quasi-orders, we have also shown that the above statement is false when ``height one'' is replaced by ``height three.'' 

At this point it may seem that, when viewed from this perspective, the results of this paper actually support Kechris's conjecture. After all, Theorem~\ref{thm:height2borel} shows that every height two, locally countable Borel partial order is Borel reducible to Turing reducibility. Doesn't this suggest that the same may be true for quasi-orders? However, upon further consideration, this argument is not very convincing. In a quasi-order of height one, it is possible to have an infinitely long sequence of distinct elements which are all related to each other, but this is impossible in a partial order of finite height. The existence of such sequences seems to cause major problems for the construction we used in the proof of Theorem~\ref{thm:height2borel}. Thus we make the following conjecture.

\begin{conjecture}
There is a locally countable Borel quasi-order of height one which is not Borel reducible to Turing reducibility.
\end{conjecture}

\bibliographystyle{alpha}
\bibliography{bibliography}

\end{document}